\documentclass[a4paper, 11pt, reqno]{amsart}
\usepackage[utf8]{inputenc}
\usepackage[usenames,dvipsnames]{color}
\usepackage{amssymb, amsmath, latexsym, graphics, graphicx,ytableau}

\makeatother

\newcommand{\ol}{\overline}

\usepackage[utf8]{inputenc}
\usepackage{amsfonts}
\usepackage{amsthm}
\usepackage{enumerate}

\newtheorem{thm}{Theorem}[section]
\newtheorem{cor}[thm]{Corollary}
\newtheorem{lem}[thm]{Lemma}

\theoremstyle{definition}

\newtheorem{example}[thm]{Example}

\newcommand{\N}{\mathbb{N}}
\newcommand{\R}{\mathbb{R}}

\newcommand{\T}{\mathbb{T}}
\newcommand{\Pb}{\mathbb{P}}
\newcommand{\cA}{\mathcal{A}}
\newcommand{\cB}{\mathcal{B}}

\newcommand{\cS}{\mathcal{S}}

\newcommand{\lanran}[1]{\left\langle #1 \right\rangle}

\title[Identities of Tropical Matrices and Plactic Monoids]{Identities of Tropical Matrices and \\ Plactic Monoids}
\begin{document}
\date{\today}
\keywords{identities, varieties, tropical matrices, plactic monoids, upper triangular matrix semigroups}
\maketitle
\begin{center}
THOMAS AIRD\footnote{Email \texttt{Thomas.Aird@manchester.ac.uk}.} \\ \ \\
Department of Mathematics, University of Manchester, \\
Manchester M13 9PL, UK.

\end{center}

\begin{abstract}
We study semigroup varieties generated by full and upper triangular tropical matrix semigroups and the plactic monoid of rank 4. We prove that the upper triangular tropical matrix semigroup $UT_n(\T)$ generates a different semigroup variety for each dimension $n$. We show a weaker version of this fact for the full matrix semigroup: full tropical matrix semigroups of different prime dimensions generate different semigroup varieties. For the plactic monoid of rank 4, $\Pb_4$, we find a new set of identities satisfied by $\Pb_4$ shorter than those previously known, and show that the semigroup variety generated by $\Pb_4$ is strictly contained in the variety generated by $UT_5(\T)$.
\end{abstract}

\par The tropical semiring and matrices over the tropical semiring are of significant interest and have been widely studied in many areas of mathematics (see for example \cite{Con,Geom}). One key use of matrices over the tropical semiring is that they admit faithful representations of semigroups which cannot be faithfully represented by matrices over fields. For example, Izhakian and Margolis showed that the bicyclic monoid has a faithful representation in the semigroup of 2 by 2 upper triangular tropical matrices \cite{2by2}, but it does not have a representation over a field. Similarly, Johnson and Kambites showed that the plactic monoid of rank $n$ can be represented by $2^n$ by $2^n$ upper triangular tropical matrices.

\par In recent years, there has been considerable of interest in the semigroup identities satisfied by tropical matrices. Izhakian \cite{Errata,Tri,New} and Okini\'nski \cite{Jan} showed that for every $n$ the semigroup of $n$ by $n$ triangular tropical matrices satisfies a non-trivial semigroup identity. Building on the work of Shitov \cite{Shitov}, Izhakian and Merlet then showed that for every $n$ the full tropical matrix semigroup of dimension $n$ satisfies a non-trivial semigroup identity \cite{Zur}.
Johnson and Kambites showed that the plactic monoid of rank $n$ satisfies all semigroup identities satisfied by $UT_d(\T)$ where $d$ is the integer part of $\frac{n^2}{4}+1$ and every semigroup identity satisfied by $UT_n(\T)$ is satisfied by $\Pb_n$ \cite{JKPlactic}.

\par In this paper we prove a conjecture posed by Johnson and Kambites. That is, we show that for every positive integer $n$ there is a semigroup identity satisfied by $UT_n(\T)$ but not by $UT_{n+1}(\T)$ \cite[Conjecture 3.5]{JKPlactic}. Moreover, Johnson and Kambites also asked whether the variety generated by $\Pb_4$ is equal to that generated by $UT_4(\T)$ and/or that generated by $UT_5(\T)$ \cite[Question 4.8]{JKPlactic}, and in Section \ref{PlacticSec} we show that $\Pb_4$ satisfies semigroup identities not satisfied by $UT_5(\T)$ and hence, the variety generated by $\Pb_4$ is strictly contained in the variety generated by $UT_5(\T)$. It is known that the variety generated by $UT_2(\T)$ is equal to the variety generated by $\Pb_2$ and similarly  the variety generated by $UT_3(\T)$ is equal to the variety generated by $\Pb_3$. It remains open if the variety generated by $\Pb_4$ is equal to the variety generated by $UT_4(\T)$.

\par In addition to this introduction, this paper comprises 4 sections. In Section 1, we introduce some notations and definitions that we use throughout the rest of the paper.
\par In Section 2 we introduce a necessary requirement for a semigroup identity to be satisfied by the semigroup of $n$ by $n$ upper triangular matrices, $UT_n(\T)$. We then use this to show that for all $n \in \N$ we can construct semigroup identities satisfied by $UT_n(\T)$ but not $UT_{n+1}(\T)$ proving the conjecture given by Johnson and Kambites \cite[Conjecture 3.5]{JKPlactic}.
\par In Section 3, we turn our attention to the full matrix semigroup, $M_n(\T)$. We show that there exists a semigroup identity satisfied by $M_3(\T)$ but not $M_4(\T)$ and more generally show that there exists a semigroup identity satisfied by $M_{p-1}(\T)$ but not $M_p(\T)$ when $p$ is prime. The question of if $M_{p-1}(\T)$ and $M_p(\T)$ generate different varieties for non-prime $p$ remains open.
\par In Section 4, we look at the plactic monoid and find a new set of semigroup identities that is satisfied by $\Pb_4$ but not by $UT_5(\T)$, partially answering the question posed by Johnson and Kambites \cite[Question 4.8]{JKPlactic} by showing that the variety generated by $\Pb_4$ is strictly contained in the variety generated by $UT_5(\T)$. 
\par \textbf{Acknowledgements.} The author thanks Marianne Johnson and Mark Kambites for helpful conversations and comments on the draft.

\section{Preliminaries}
Given an alphabet $\Sigma$, we denote by $\Sigma^*$, the set of all words over $\Sigma$ and by $\Sigma^+$ the set of non-empty words over $\Sigma$. We write $l(w)$ for the \emph{length} of a word $w$. For a word $w$, we write $w_{(i)}$ to denote the $i$th letter of $w$ and write ${|w|}_a$ for the number of times the letter $a \in \Sigma$ appears in $w$. For $u,v \in \Sigma^+$ we say that $u$ is:  a \emph{suffix} of $v$ if there exists $v_1 \in \Sigma^*$ such that $v_1u = v$, a \emph{prefix} of $v$ if there exists $v_2 \in \Sigma^*$ such that $uv_2 = v$, and a \emph{factor} of $v$ if there exists $v_1,v_2 \in \Sigma^*$ such that $v = v_1uv_2$. 
A \emph{semigroup identity} is a pair of (non-empty) distinct words, denoted $``u = v"$. We say that a semigroup $S$ satisfies the identity $u=v$ if every morphism from $\Sigma^+$ to $S$ maps $u$ and $v$ to the same element of $S$.
The variety generated by a semigroup $S$ is the class of all semigroups that satisfy all the semigroup identities satisfied by $S$.
Let $S$ be a semigroup and $w \in \{a,b\}^+$; then for $x,y \in S$, we write $w(x,y)$ to denote the evaluation of $w$ in $S$ obtained by preforming the substitution $a \mapsto x$ and $b \mapsto y$. In the case where $S=\{a,b\}^+$ we write $w[x,y]$, rather than $w(x,y)$, to indicate that $w[x,y]$ is again a word in $\{a,b\}^+$.

\par The \emph{tropical semiring}, denoted $\T$, is the set of real numbers augmented with $-\infty$ defined with two associative binary operations, maximum as its addition and addition as its multiplication. This forms a semiring structure with addition distributing over maximum. As $\T$ is a semiring, for any $n \in \N$, we can construct the (non-commutative when $n > 1$) semiring of $n$ by $n$ matrices over $\T$, which we denote by $M_n(\T)$. However, we will only consider the multiplicative structure of $M_n(\T)$ and therefore consider $M_n(\T)$ to be a semigroup. Similarly we define $UT_n(\T)$, to be the subsemigroup of $M_n(\T)$ of all $n$ by $n$ upper triangular matrices over $\T$ with $-\infty$'s below the diagonal. When writing matrices we will use blank entries for $-\infty$ when it is clearer.
\par For a matrix $A = (A_{ij}) \in M_n(\T)$, we write $G_A = (V,E)$ for the weighted digraph associated to $A$, that is, the digraph with vertex set $V(G_A) = \{1,\dots,n\}$ and edge set $E(G_A)$ containing directed edges $(i,j)$ weighted by $A_{ij}$ if and only if $A_{ij} \neq -\infty$.
Similarly, for $A,B \in M_n(\T)$, we write $G_{A,B}$ for the labelled-weighted digraph with vertex set $\{1,\dots,n\}$ and edge set $E(G_A) \cup E(G_B)$. We call $G_{A,B}$ the compound digraph associated to $A$ and $B$. Moreover, an edge in $G_{A,B}$ is labelled by $A$ (respectively, by $B$) if it came from $E(G_{A})$ (respectively, $E(G_B)$. An edge $(i,j)$ labelled by $A$ (respectively, labelled by $B$) is weighted with the matrix entry $A_{ij}$ (respectively, $B_{ij}$).

\par A path $\gamma$ on a digraph is a series of edges $(i_1,j_1),\dots,(i_m,j_m)$ such that $j_k=i_{k+1}$ for all $1 \leq k < m$. We say $g$ is a \emph{node} of $\gamma$ if an edge starting or ending at $g$ is in $\gamma$, and call an edge a \emph{loop} if it starts and ends at the same node. 
A path $\gamma$ is said to have \emph{length} $m$ if $\gamma$ contains $m$ edges (counted with multiplicity), written $l(\gamma) = m$, and has \emph{simple length} $m$ if $\gamma$ contains $m$ \emph{non-loop} edges (again counted with multiplicity). A path is called \emph{simple} if it does not contain any loops. For any word $w \in \{A,B\}^+$ and $\gamma$ a path in $G_{A,B}$, we say $\gamma$ \emph{is labelled} $w$ if $l(\gamma) = l(w)$ and, for all $1 \leq r \leq l(\gamma)$, the edge $(i_{r}, j_r)$ is labelled $w_{(r)}$, the $r$th letter of $w$.

\par For $n \geq 1$ we define the plactic monoid of rank $n$, $\Pb_n$, to be the monoid generated by the set $\{1,\dots,n\}$ with the Knuth relations:
\begin{align*}
bca &= bac \text{ for } a < b \leq c \\
cab &= acb \text{ for } a \leq b < c.    
\end{align*}
There exist a combinatorial way describe $\Pb_4$ where each element of $\Pb_4$ corresponds to a semi-standard Young tableaux, that is, a Young diagram with numbers less than or equal to $n$ such that the columns are strictly decreasing and the rows are weakly increasing, and multiplication of Young tableau is then given by the \emph{Schensted’s insertion algorithm}. For the interested reader, the authors of \cite{JKPlactic} provide more on this viewpoint.
\section{Upper Triangular Matrix Semigroups}
In this chapter we restrict our attention to the subsemigroup of upper triangular tropical matrices and show that upper triangular tropical matrix semigroups of different dimensions generate different semigroup varieties.
\begin{lem} \label{factor}
Suppose $u,v,w \in \{a,b\}^*$ are words such that $w$ has length $n$ and is a factor of $u$ but not $v$. Then there exists $A,B \in UT_{n+1}(\T)$ such that $u(AB,BA) \neq v(AB,BA)$.
\end{lem}
\begin{proof}
Let $w \in \{a,b\}^*$ be a word of length $n$. We recursively define $n+1$ parameters $c_1,\cdots,c_{n+1} \in \T,$ using the structure of the word $w$. Let $c_1 = 0$ and for $2 \leq k \leq n+1$ let
\[
c_k =
\begin{cases}
c_{k-1} - 1 \text{ if } w_{(k-1)} = a \\
c_{k-1} + 1 \text{ if } w_{(k-1)} = b.
\end{cases}
\]
\par From these parameters we can then define matrices $A_w,B_w \in UT_{n+1}(\T)$ to be
\[
A_w =
\begin{pmatrix}
c_1 & & & \\
& c_2 & &  \\
& & \ddots & \\
& & & c_{n+1}
\end{pmatrix}
B_w =
\begin{pmatrix}
-c_1 & 0 & & &\\
&  -\infty & \ddots & & \\[0.5ex]
& &  \ddots & \ddots & & \\[1.5ex]
& & &  -\infty & 0 \\[0.5ex]
& & & &  -c_{n+1}
\end{pmatrix}
\]
where $(A_w)_{kk} = c_{k}$ for $1 \leq k \leq n+1 $ and $-\infty$ otherwise; $(B_w)_{11} = -c_1$, $(B_w)_{n+1,n+1} = -c_{n+1}$, $(B_w)_{k,k+1} = 0$ for $1 \leq k \leq n$ and $-\infty$ otherwise. 
\par Letting $A = A_w$ and $B = B_w$. We aim to show that if $w$ is a subword of $u$ but not $v$, then we get that $u(AB,BA) \neq v(AB,BA)$.
Note that $AB$ and $BA$ are given by the following matrices
\[ 
AB =
\begin{pmatrix}
0 & c_1 & & &\\[0.5ex]
& -\infty & c_2 & & \\[0.5ex]
& & \ddots & \ddots & \\[0.5ex]
& & & -\infty & c_{n} \\[0.5ex]
& & & & 0
\end{pmatrix}
BA =
\begin{pmatrix}
0 & c_2 & & &\\[0.5ex]
& -\infty & c_3 & & \\[0.5ex]
& & \ddots & \ddots & \\[0.5ex]
& & & -\infty & c_{n+1} \\[0.5ex]
& & & & 0 
\end{pmatrix}.
\]
\par Consider the labelled-weighted digraph $G_{AB,BA}$;
nodes 1 and $n+1$ each have two loops of weight 0 labelled $AB$ and $BA$ and for each $1 \leq i \leq n$ there are two edges from $i$ to $i+1$ of weight $c_i$ and $c_{i+1}$ labelled $AB$ and $BA$ respectively.
Moreover, we define a function $f_w$ by
\[f_w: \{a,b\}^* \rightarrow \T, \quad t \mapsto t(AB,BA)_{1,n+1}\]
which corresponds to the maximum weight of a path labelled by $t$ from node 1 to $n+1$.
\par By construction, we have that $c_i > c_{i+1}$ if $w_{(i)} = a$ and $c_i < c_{i+1}$ if $w_{(i)}=b$. Hence, the weight of any path from $1$ to $n+1$ is bounded above by the weight of the unique path $\rho$ of length $n$ which takes the edge of largest weight from $i$ to $i+1$ for each $1 \leq i \leq n$. Moreover, $\rho$ is labelled $w$, and hence the upper bound is $f_w(w)$. So for any word $t$, we have that $f_w(t) \leq f_w(w)$. If $t = sws'$ is a word containing $w$ as a factor, a path of maximal weight is labelled $s$ around the loops at 1, $w$ along $\rho$, and $s'$ around the loops at $n+1$, gives a path of weight $f_w(w)$, and hence $f_w(t) = f_w(w)$. On the other hand, if $t$ does not contain $w$ as a factor, then a path from 1 to $n+1$ labelled $t$ cannot contain the simple path $\rho$. It follows that at some step of path we must traverse a non-maximal weight edge between two consecutive nodes. Thus, $f_w(t) < f_w(w)$ in this case.
\par Therefore, $f_w(u) = f_w(w) > f_w(v)$ as $w$ is a factor of $u$ but not $v$. Hence, letting $A= A_w$ and $B= B_w$ we have that there exists $A,B \in UT_{n+1}(\T)$ such that $u(AB,BA) \neq v(AB,BA)$.
\end{proof}
The following corollary is an direct implication of a theorem of Izhakian \cite[Theorem 4.8]{Errata}, by noticing that $AB$ and $BA$ have the same diagonal entries for all $A,B \in UT_n(\T)$. This gives us a way of generating semigroup identities for $UT_n(\T)$.

\begin{cor} \label{Zur}
Let $w \in \{a,b\}^+$ be any word having as its factors all the words of length $n-1$ such that $waw$ and $wbw$ have no letter appearing $n$ times sequentially. Then, the semigroup identity
\[w a w[ab,ba] = w b w[ab,ba]\]
is satisfied by $UT_n(\T)$.
\end{cor}

\begin{example} \label{example}
For $n=3$, $w=ab^2a^2b$ has all words of length $2$ as a factor, and neither $waw$ nor $wbw$ have $a^3$ or $b^3$ as factors. Therefore, by the above theorem, $waw[ab,ba] = wbw[ab,ba]$ is an identity that holds in $UT_3(\T)$. We will use this example later in the paper.
\end{example}

\begin{thm} \label{utvar}
For all $n \in \N$, there exists an identity satisfied by $UT_n(\T)$ but not satisfied by $UT_{n+1}(\T)$.
\end{thm}
\begin{proof}
As matrix multiplication is commutative if and only if $n=1$, the identity $ab=ba$ is satisfied by $UT_1(\T)$, but not $UT_2(\T)$.
It is known \cite{2by2} that $UT_2(\T)$ satisfies the Adjan identity, $ab^2a^2bab^2a = ab^2aba^2b^2a$. Note that this identity can be written in the form $u[ab,ba] = v[ab,ba]$ where $u = abaab$ and $v=abbab$, and since $a^2$ is a factor of $u$ but not $v$ then the identity is falsified in $UT_3(\T)$ by Lemma~\ref{factor}.
\par For $n=3$, we can see by Example~\ref{example}, that the following identity
\[u_3 :=ab^2a^2b b  ab^2a^2b[ab,ba] = ab^2a^2b a  ab^2a^2b[ab,ba] =: v_3\]
is satisfied by $UT_3(\T)$.
Note that $bab$ is a factor of $u_3$ but not $v_3$. Thus, $u_3 = v_3$ is falsified in $UT_4(\T)$ by Lemma~\ref{factor}.
\par Now let $n \geq 4$ and define $w$ to be the word of length $n$ given by $w = a\tilde{w}$ where
\[ \tilde{w} =
\begin{cases}
b(ab)^\frac{n-2}{2} \text{ if } n \text{ is even,} \\
(ba)^\frac{n-1}{2} \text{ if } n \text { is odd.}
\end{cases}\]
\par We aim to construct a word $\ol{w} \in \{a,b\}^*$ such that for $u =\ol{w}a\ol{w}$ and $v = \ol{w}b\ol{w}$ we have that: $u$ and $v$ do not have any letter appearing $n$ times sequentially, the word $\ol{w}$ contains sufficiently many factors for Theorem~\ref{Zur} to apply, so that the identity $u[ab,ba] = v[ab,ba]$ is satisfied by $UT_n(\T)$, and that the word $w$ is a factor of $u$ but not of $v$, so that $u(AB,BA) \neq v(AB,BA)$ for some $A,B \in UT_{n+1}(\T)$ by Lemma~\ref{factor}.
\par Let $w_{1}, \dots, w_{m}$ be a complete list of words of length $n-1$ taken in some arbitrary but fixed order. We now define $w'_{1}, \dots, w'_{m}$ depending on if $n$ is even or odd in the following way:
\par If $n$ is even, let $w'_{i}$ be equal to the (possibly empty) word obtained from $w_{i}$ by removing the prefix $b^2$ if possible, and the suffix $a^2$ if possible. 
If $n$ is odd, let $w'_{i}$ be equal to the word obtained from $w_{i}$ by removing the prefix $b^2$ if possible, and the suffix $b^2$ if possible.
\par Now, we define 
\[\ol{w} = 
\begin{cases}
\tilde{w}ba(b^2w'_{1}a^2)(b^2w'_{2}a^2) \cdots (b^2w'_{m-1}a^2)(b^2w'_{m}a^2)b^2a \text{ if } n \text{ is even,}\\
\tilde{w}(b^2w'_{1}b^2)a(b^2w'_{2}b^2)\cdots (b^2w'_{m-1}b^2)a(b^2w'_{m}b^2)a \text{ if } n \text{ is odd.}
\end{cases}
\]
By construction, $\ol{w}$ clearly contains each word of length $n-1$ as a factor. Recall that, $\tilde{w}$ is an alternating product of $bab\cdots$, so it does not contain $a^n$ or $b^n$ as a factor. Likewise, by construction each of the bracketed expressions $(b^2w'_{i}a^2)$ and $(b^2w'_{i}b^2)$ do not contain $a^n$ or $b^n$ as a factor as $n \geq 4$. 
Similarly, it can be seen that $\ol{w}$ does not contain $a^n$ or $b^n$. Furthermore, since $\ol{w}$ begins and ends with $ba$, it follows that $u= \ol{w}a\ol{w}$ and $v = \ol{w}b\ol{w}$ do not contain $a^n$ or $b^n$. This shows that Theorem~\ref{Zur} applies, so that $u[ab,ba] = v[ab,ba]$ is satisfied by $UT_n(\T)$.
\par Moreover, we can see that $w$ is a factor of $u = \ol{w}a\ol{w}$ as $a\ol{w} = a\tilde{w}\cdots = w\cdots$ but not a factor of $v$. Thus, by Lemma~\ref{factor}, $u[ab,ba] = v[ab,ba]$ is falsified in $UT_{n+1}(\T)$.
\end{proof}
Another possible approach to Theorem~\ref{utvar} would be to use knowledge about the free objects in these varieties discussed by Kambites in \cite{FreeObjects} to show that the free objects in the varieties generated by $UT_n(\T)$ and $UT_{n+1}(\T)$ are non-isomorphic for all $n \in \N$.

\section{Full Matrix Semigroup}
We introduce the notation that $\ol{n} = \mathrm{lcm}\{1, \dots, n\}$ and write $\lanran{u,v}[x,y]$ for the semigroup identity $u[x,y] = v[x,y]$. Moreover, we say a matrix $A \in M_n(\T)$ has the underlying permutation of $\sigma \in \cS_n$ if $A_{ij} \neq -\infty$ if and only if $j = \sigma(i)$. A matrix is invertible if and only if it has an underlying permutation \cite{Invert}. The following theorem of Izhakian and Merlet allows us to produce semigroup identities for $M_n(\T)$.

\begin{thm}\cite[Theorem 3.6]{Zur} \label{FullIden}
For any $t \geq (n -1)^2 + 1$ and any identity $u = v$ which holds in $M_{n-1}(\T)$, where $u, v \in \{a, b\}^*$, the following holds:
\begin{enumerate}[(i)]
    \item If $q=r$ is an identity which holds in $UT_n(\T)$, then $M_n(\T)$ satisfies the identity
\[ \lanran{ua, va}\left[ (qr)^t\left[a^{\ol{n}},b^{\ol{n}}\right],(qr)^tr\left[a^{\ol{n}},b^{\ol{n}}\right]\right] \]
where $p, q, r \in \{a, b\}^*$.
    \item If $p{q}p = p{r}p$ is an identity which holds $UT_n(\T)$, then $M_n(\T)$ satisfies the identity given by
\[ \lanran{ua, va}\left[ w{q}p\left[a^{\ol{n}},b^{\ol{n}}\right],w{r}p\left[a^{\ol{n}},b^{\ol{n}}\right]\right] \]
where $w = (p{q}p{r}p)^t$ and $p, q, r \in \{a, b\}^*$.
\end{enumerate}
\end{thm}

\begin{thm}
There exists an identity satisfied by $M_3(\T)$ that is not satisfied by $M_4(\T)$.
\end{thm}
\begin{proof} 
We apply Theorem~\ref{FullIden} in the case $n=3$, note that \cite[Remark 3.5]{Zur} allows us to omit the exponent $t$ in this case.
Set $u= a^2b^3a^3babab^3a^2$ and $v = a^2b^3ababa^3b^3a^2$. Then $u=v$ holds in $M_2(\T)$ by \cite{m2identity}.
Now set $p = ab^2a^2b[ab,ba]$, $q=ab$ and $r=ba$.
Then $pqp = prp$ holds in $UT_3(\T)$ by Example~\ref{example}. Note that $\ol{n} = 6$ when $n = 3$, Theorem~\ref{FullIden}(ii) now yields the identity of length 5832 satisfied by $M_3(\T)$
\[ \lanran{s,t} :=\lanran{ua, va}\left[ wqp\left[a^6,b^6\right],wrp\left[a^6,b^6\right]\right]  \]
where $w = p{q}p{r}p$.

Now, let $X,Y \in M_4(\T)$ be given by
\[X =
\begin{pmatrix}
2 & -\infty & -\infty & -\infty \\
-\infty & 4 & -\infty & -\infty \\
3 & -\infty & -\infty & -\infty \\
-\infty & -\infty & 4 & 0   
\end{pmatrix} \quad
Y =
\begin{pmatrix}
-\infty & 0 & -\infty & -\infty \\
-\infty & -\infty & 1 & -\infty \\
-\infty & -\infty & -\infty & 1 \\
1 & -\infty & -\infty & -\infty 
\end{pmatrix}
\]
Then, a computation (run on the GAP computer algebra system \cite{gap}) gives $s(X,Y) \neq t(X,Y)$ and hence we have constructed an identity satisfied by $M_3(\T)$ but not by $M_4(\T)$.
\end{proof}

\begin{lem} \label{false}
Let $n \geq 3$ be odd, and $A,B \in M_n(\T)$ be invertible matrices such that $A$ has the underlying permutation of an $n$-cycle and $B$ is diagonal and not a scaling of the identity matrix. Then, there exists an identity satisfied by $M_2(\T)$, $u_2=v_2$, such that $u_2(A,B) \neq v_2(A,B)$.
\end{lem}

\begin{proof}
Let $u_2= a^2 b^4a^2\ a^2b^2\ a^2b^4a^2$, $v_2 = a^2 b^4a^2\ b^2a^2\ a^2 b^4a^2$. Then $u_2 = v_2$ is an identity satisfied in $M_2(\T)$ by \cite[Theorem 3.9]{2by2}.
Now, let $A,B \in M_n(\T)$ be such that $A$ is has the underlying permutation of an $n$-cycle $\sigma$ and $B$ is a diagonal matrix. As $A$ and $B$ are invertible matrices, they are cancellative and hence $u_2(A,B) = v_2(A,B)$ if and only if $A^2B^2 = B^2A^2$, by cancelling $A^2B^4A^2$ from both sides of $u_2(A,B)$ and $v_2(A,B)$. However, 
\begin{align*}
(A^2B^2)_{i,\sigma^2(i)} &=
A_{i\sigma(i)}A_{\sigma(i)\sigma^2(i)}B^2_{\sigma^2(i),\sigma^2(i)}, \\
(B^2A^2)_{i,\sigma^2(i)} &= 
B_{ii}^2A_{i\sigma(i)}A_{\sigma(i)\sigma^2(i)}.
\end{align*}
Moreover, as $(A^2B^2)_{ij} = -\infty = (B^2A^2)_{ij}$ if $j \neq \sigma^2(i)$, we get that $A^2B^2 = B^2A^2$ if and only if $B_{ii} = B_{\sigma^2(i),\sigma^2(i)}$ for all $1 \leq i \leq n$. That is, as $\sigma$ is an $n$-cycle and $n$ is odd, if and only if $B$ is a scaling of the identity matrix. Therefore, $u_2(A,B) \neq v_2(A,B)$ if $B$ is a diagonal matrix and not a scaling of the identity matrix.
\end{proof}

\begin{lem} \label{induct}
For each $k$ in the range $3 \leq k \leq n$, let $q_k = r_k$ be an identity satisfied by $UT_k(\T)$ over $\{a,b\}$, and let $t_k$ be a fixed integer with $t_k \geq (k-1)^2 + 1$. Let $A_n = a$, $B_n = b$, and for $k=n, \ldots, 3$ recursively define
\[ A_{k-1} = (q_kr_k)^{t_k}\left[A_k^{\ol{k}},B_k^{\ol{k}}\right] \text{ and } B_{k-1} =(q_kr_k)^{t_k}r_k\left[A_k^{\ol{k}},B_k^{\ol{k}}\right]. \]
Then, for any identity satisfied by $M_2(\T)$, $u_2=v_2$, we have that
\[u_2[A_2,B_2]A_2A_3\cdots A_{n-1} = v_2[A_2,B_2]A_2A_3 \cdots A_{n-1}.\]
is an identity satisfied by $M_n(\T)$.
\end{lem}
\begin{proof}
For each $3 \leq k \leq n$, we construct the identity $u_k = v_k$ which holds in $M_k(\T)$ using Theorem~\ref{FullIden}(i), as follows
\begin{align*}
u_k &= (u_{k-1}a)\left[ (q_kr_k)^{t_k}\left[a^{\ol{k}},b^{\ol{k}}\right],(q_kr_k)^{t_k}r_k\left[a^{\ol{k}},b^{\ol{k}}\right]\right] \\
v_k &= (v_{k-1}a)\left[ (q_kr_k)^{t_k}\left[a^{\ol{k}},b^{\ol{k}}\right],(q_kr_k)^{t_k}r_k\left[a^{\ol{k}},b^{\ol{k}}\right]\right].
\end{align*}
\par By expressing $A_{k-1}$ as $a[A_{k-1},B_{k-1}]$, and substituting the definitions of $A_{k-1},B_{k-1}$ and the definition of $u_k$, we have that the following equalities hold for $3 \leq k \leq n$
\begin{align*}
u_{k-1}[A_{k-1},B_{k-1}]A_{k-1}A_k\cdots A_{n-1} &= (u_{k-1}a)[A_{k-1},B_{k-1}]A_k \cdots A_{n-1} \\
&= (u_{k-1}a)\Big[(q_kr_k)^{t_k}\big[A_k^{\ol{k}},B_k^{\ol{k}}\big], \\
&\qquad \qquad (q_kr_k)^{t_k}r_k\big[A_k^{\ol{k}},B_k^{\ol{k}}\
]\Big]A_{k}\cdots A_{n-1} \\
&=u_k[A_k,B_k]A_k\cdots A_{n-1}.
\end{align*}
where the product $A_{k}\cdots A_{n-1}$ is taken to be the empty word when $k=n$. Similarly, it can be easily seen that
\[ v_{k-1}[A_{k-1},B_{k-1}]A_{k-1}\cdots A_{n-1} = v_k[A_k,B_k]A_k\cdots A_{n-1}. \]
for $3 \leq k \leq n$. So, through the equalities given above, we have that
\begin{align*}
u_2[A_2,B_2]A_2A_3\cdots A_{n-1} &= u_n[A_n,B_n], \text{ and} \\
v_2[A_2,B_2]A_2A_3\cdots A_{n-1} &= v_n[A_n,B_n].
\end{align*}
Thus, as $u_n = v_n$ is an identity satisfied by $M_n(\T)$, we have that the identity $u_2[A_2,B_2]A_2A_3\cdots A_{n-1} =
v_2[A_2,B_2]A_2A_3\cdots A_{n-1}$ is satisfied by $M_n(\T)$.
\end{proof}

\begin{thm}
Let $p$ be a prime. Then there exists an identity satisfied by $M_{p-1}(\T)$ but not by $M_p(\T)$.
\end{thm}
\begin{proof}
As matrix multiplication is not commutative for $n > 1$, $ab = ba$ is satisfied by $M_1(\T)$ but not by $M_2(\T)$ and
by Lemma~\ref{false} there exists an identity satisfied by $M_2(\T)$ but not by $M_3(\T)$.
\par Suppose then that $p$ is a prime greater than 3. Let $X \in M_p(\T)$ be a permutation matrix of a $p$-cycle $\sigma$ and $Y \in M_p(\T)$ be an invertible diagonal matrix where all entries on the diagonal are different. 
For each $3 \leq k < p$ choose and identity $q_k = r_k$ satisfied by $UT_k(\T)$ with the property that the number of occurrences of $a$ in each side is congruent to $-1 \mod p$. 
This can clearly be done by starting with any identity satisfied by $UT_k(\T)$, as the letter $a$ must appear the same number of times on each side, and then appending a power of $a$ to both sides on the right.
We let $A_{p-1} = a$, $B_{p-1} = b$ and define words $A_{k-1},B_{k-1}$ for $3 \leq k < p$ recursively, as in Lemma~\ref{induct}, by 
\begin{equation*}
A_{k-1} = (q_kr_k)^{t}\left[A_{k}^{\ol{k}},B_{k}^{\ol{k}}\right] \text{ and } B_{k-1} =(q_kr_k)^{t}r_k\left[A_{k}^{\ol{k}},B_{k}^{\ol{k}}\right]
\end{equation*}
where $t = \frac{p^3-1}{2}$. Note that $t \geq (k-1)^2 + 1$ for every $3 \leq k < p$.
\par For $2 \leq m \leq p-1$, let $\cA_m =A_m(X,Y) \in M_p(\T)$ and $\cB_m = B_m(X,Y) \in M_p(\T)$. We now show $\cA_k$ has the underlying permutation of a $p$-cycle and that we can choose $q_k = r_k$ such that  $\cB_k$ is an invertible diagonal matrix in which all the entries on the diagonal are different for every $2 \leq k \leq p-1$. This is true for $\cA_{p-1}$ and $\cB_{p-1}$ by definition.
So, for induction, suppose it is true for $\cA_k$ and $\cB_k$ and we show it is true in the $k-1$ case.
\par Let $u \in \{a,b\}^*$. As $\cA_k$ and $\cB_k$ are invertible, the matrix $u(\cA_k,\cB_k)$ is also invertible.
Moreover, as the underlying permutation of $\cA_k$ is a $p$-cycle, and $\cB_k$ is a diagonal matrix, it follows that the underlying permutation of $u(\cA_k,\cB_k)$ depends only on the number of occurrences of $a$ in $u$ modulo $p$.
If ${|u|}_a \equiv 0 \mod p$, then $u(\cA_k,\cB_k)$ is a diagonal matrix; otherwise it is a $p$-cycle as $\sigma$ is a $p$-cycle and $p$ is prime, so $\sigma^s$ is a $p$-cycle unless $p$ divides $s$.
\par Letting ${|A_k|}_a = n$, we can see that
\begin{align*}
{|A_{k-1}|}_a = {|(q_kr_k)^{t}|}_a {|A_k^{\ol{k}}|}_a &\equiv t{|q_kr_k|}_a n\ol{k} \equiv \frac{p^3-1}{2}(-2)n\ol{k} \equiv n\ol{k} \mod p \\
{|B_{k-1}|}_a = {|(q_kr_k)^{t}r_k|}_a {|A_k^{\ol{k}}|}_a &\equiv {|(q_kr_k)^{t}|}_a n\ol{k} + {|r_k|}_a n \ol{k} \equiv n\ol{k} -n\ol{k}\equiv 0 \mod p
\end{align*}
Thus, as $p$ does not divide $n$ as $\cA_k$ is a $p$-cycle, and $p$ does not divide $\ol{k}$, $p$ does not divide $n\ol{k}$ and hence $\cA_{k-1}$ has the underlying permutation of a $p$-cycle and $\cB_{k-1}$ is a diagonal matrix as required. If not every entry on the diagonal of $\cB_{k-1}$ is different, then we can replace our choice of the identity $q_k = r_k$ 
by multiplying both sides on the right by a sufficiently high power of $b$ we can ensure that all entries are different. We can see this, as each additional $b$ adds a $(\cB_k^{\ol{k}})_{ii}$ to $(\cB_{k-1})_{ii}$ and by assumption every $(\cB_k^{\ol{k}})_{ii}$ is different, so after sufficiently many $b$'s every entry on the diagonal will be different.
\par Therefore, we have shown that $\cA_{k-1}$ has the underlying permutation of a $p$-cycle and $\cB_{k-1}$ is a diagonal matrix which is not the scaling of the identity matrix. 
So, by induction, $\cA_2$ has the underlying permutation of a $p$-cycle and $\cB_2$ is a diagonal matrix which is not a scaling of the identity. Therefore, if we let $u_2 = v_2$ be the identity satisfied by $M_2(\T)$ given by Lemma~\ref{false}, $u_2(\cA_2,\cB_2) \neq v_2(\cA_2,\cB_2)$ and thus
\[
u_2(\cA_2,\cB_2)\cA_2\cdots \cA_{p-2} \neq v_2(\cA_2,\cB_2)\cA_2\cdots \cA_{p-2}
\]
as $\cA_2,\dots,\cA_{p-2} \in M_p(\T)$ are invertible matrices and therefore cancellative. However, by Lemma~\ref{induct},
\[u_2[A_2,B_2]A_2\cdots A_{p-2} = v_2[A_2,B_2]A_2\cdots A_{p-2} \] is an identity satisfied by $M_{p-1}(\T)$ and so we have constructed an identity satisfied by $M_{p-1}(\T)$ that is falsified by $X,Y \in M_p(\T)$.
\end{proof}

\section{Plactic Monoid} \label{PlacticSec}
In this section we show the plactic monoid of rank 4, $\Pb_4$, does not generate the same variety as $UT_5(\T)$. To do this we will use the faithful tropical representation of $\Pb_4$ given in \cite{JKPlactic}. We begin by recalling some notation used in the definition of this representation. We define $[n]:= \{1,\dots,n\}$. For $S,T \in 2^{[n]}$, we write $S^i$ for the $i$th smallest element of $S$, and say $S \leq T$ if $|S| \geq |T|$ and $S^i \leq T^i$ for each $i \leq |T|$. Moreover, for $P,Q \in 2^{[n]}$, we write $[P,Q]$ for the order interval from $P$ to $Q$, and $\cup[P,Q]$ for the union of sets in the order interval. 
\par The following theorem is given in greater generality in \cite{JKPlactic}, but we only require $n=4$ in what follows.
\begin{thm}\cite[Theorem 2.8]{JKPlactic} \label{placmap}
There exists a faithful semigroup morphism $\rho: \Pb_4 \rightarrow UT_{2^{[4]}}(\T)$, where
\[ \rho(x)_{P,Q} = 
\begin{cases}
-\infty &\text{if } |P| \neq |Q| \text{ or } P \nleq Q; \\
1 &\text{if } |P| = |Q| \text{ and } x \in \cup[P,Q];\\
0 &\text{otherwise}.
\end{cases}\]
for each generator $x \in \Pb_4$, extending multiplicatively for products of generators and defining the identity element $e$ as
\[ 
\rho(e) =
\begin{cases}
-\infty &\text{if } |P| \neq |Q| \text{ or } P \nleq Q \\
0 &otherwise.
\end{cases}
\]
\end{thm}
Note that for all $x \in \Pb_4$, $\rho(x)$ is a block matrix where the largest block is of size $6$ by $6$ and all simple paths in $G_{\rho(x)}$ have length at most 4.
\begin{lem} \label{tech2}
Let $X,Y \in UT_m(\T)$ and $u = v$ be an identity satisfied by $UT_n(\T)$ where $n \leq m$. If there exists a path in $G_{X,Y}$ of simple length less than or equal to $n-1$ of maximal weight among all paths from $i$ to $j$ labelled $u$, then $(u(X,Y))_{ij} \leq (v(X,Y))_{ij}$.
\end{lem}
\begin{proof}
Let $\gamma$ be a path of maximal weight labelled $u$ of simple length $k \leq n-1$ from $i$ to $j$. Let $\ol{X},\ol{Y} \in UT_{k+1}(\T)$ be the matrices obtained from $X$ and $Y$ by removing rows and columns not indexed by nodes visited by $\gamma$, with the rows and columns labelled by their original labelling. Clearly, $\gamma$ is a path in $G_{\ol{X},\ol{Y}}$ having maximal weight among all paths from $i$ to $j$ labelled by $u$, and so we have that $(u(X,Y))_{ij} = (u(\ol{X},\ol{Y}))_{ij}$. Moreover, we have that
\[(u(X,Y))_{ij} = (u(\ol{X},\ol{Y}))_{ij} = (v(\ol{X},\ol{Y}))_{ij} \leq (v(X,Y))_{ij} \]
where the second equality holds since $u=v$ is an identity satisfied by $UT_n(\T)$ and hence also for $UT_{k+1}(\T)$ for all $k+1 \leq n$, and where the inequality follows from the construction of $\ol{X}$ and $\ol{Y}$.
\end{proof}

\begin{thm} \label{plactic}
Let $u,v \in \{a,b\}^*$ be such that $u[ab,ba] = v[ab,ba]$ is a semigroup identity that holds in $UT_4(\T)$. Then the identity $abuab[ab,ba] = abvab[ab,ba]$ is satisfied by $\Pb_4$.
\end{thm}

\begin{proof}
Let $\rho: \Pb_4 \rightarrow UT_{2^{[4]}}(\T)$ be the morphism given in Theorem~\ref{placmap}, and let $x,y \in \Pb_4$.
Recall that $\rho$ is faithful and the semigroup identity $abuab[ab,ba] = abvab[ab,ba]$ is satisfied by $UT_4(\T)$. 
Let $X = \rho(x), Y= \rho(y)$ and consider $abuab(XY,YX)$ and $abvab(XY,YX)$. Note that for all $(P,Q) \neq (\{1,2\},\{3,4\})$ a path from $P$ to $Q$ in $G_{X,Y}$ has simple length at most 3. Hence we may apply Lemma~\ref{tech2} (in both directions) to obtain $abuab(XY,YX)_{P,Q} = abvab(XY,YX)_{P,Q}$ for all $(P,Q) \neq (\{1,2\},\{3,4\})$.
Now by the fact that $\rho$ is faithful, we have that
\begin{multline} \label{first}
abuab(xy,yx) = abvab(xy,yx) \text{ if and only if } \\ (abuab(XY,YX))_{\{1,2\},\{3,4\}} = (abvab(XY,YX))_{\{1,2\},\{3,4\}}.
\end{multline}
Therefore, it suffices to check that $abuab[ab,ba]=abvab[ab,ba]$ holds for the $\{1,2\},\{3,4\}$ entry in the image of $\rho$.

\par By the definition of $\rho$, we have that for $s \in \Pb_4$, $\rho(s)_{P,P}$ is the total number of occurrences of letters from the set $P$ in some fixed word representing $s$. It follows from this that
\[ \rho(s)_{\{1,2\},\{1,2\}}+\rho(s)_{\{3,4\}\{3,4\}} = \rho(s)_{\{1,3\},\{1,3\}}+\rho(s)_{\{2,4\},\{2,4\}} \]
for all $s \in \Pb_4$ as all words representing $s$ have the same number of 1's, 2's, 3's and 4's and so both sides of the equality counter the number of occurrences of 1's, 2's, 3's and 4's in some word representing $s$.
Then for each $s$ we have that, either
\[ \rho(s)_{\{1,3\},\{1,3\}} \geq \rho(s)_{\{1,2\},\{1,2\}} \text{ or }\rho(s)_{\{2,4\},\{2,4\}} \geq \rho(s)_{\{3,4\}\{3,4\}}.\]
\par We now look at the graph $G_{XY,YX}$. This is the graph where there is an edge from $i$ to $j$ labelled $XY$ and weighted by $\rho(xy)_{ij}$, and an edge from $i$ to $j$ labelled $YX$ and weighted by $\rho(yx)_{ij}$.
Note $\rho(xy)_{\{1,3\},\{1,3\}} \geq \rho(xy)_{\{1,2\},\{1,2\}}$ if and only if $\rho(yx)_{\{1,3\},\{1,3\}} \geq \rho(yx)_{\{1,2\},\{1,2\}}$ as $\rho(xy)$ and $\rho(yx)$ have the same diagonal entries. Suppose that $\rho(xy)_{\{1,3\},\{1,3\}} \geq \rho(xy)_{\{1,2\},\{1,2\}}$, and let $\gamma$ be a path of maximal weight in $G_{XY,YX}$ from $\{1,2\}$ to $\{3,4\}$ labelled by the word $abuab[XY,YX]$. 
\par We split into two cases
\begin{enumerate}[(i)]
\item If $\gamma$ does not contain an edge from $\{1,2\}$ to $\{1,3\}$. Then, $\gamma$ is a path of simple length $\leq 3$, so by Lemma~\ref{tech2} $(abuab(XY,YX))_{\{1,2\},\{3,4\}} \leq (abvab(XY,YX))_{\{1,2\},\{3,4\}}$.
\item If $\gamma$ contains an edge from $\{1,2\}$ to $\{1,3\}$. Then $\gamma$ is of the form
\[ \gamma = \lambda_{\{1,2\}} \circ \gamma_{\{1,2\},\{1,3\}} \circ \mu \]
where $\lambda_{\{1,2\}}$ is a path made up of loop edges around node $\{1,2\}$, $\gamma_{\{1,2\},\{1,3\}}$ is the subpath of $\gamma$ corresponding to an edge from $\{1,2\}$ to $\{1,3\}$ and $\mu$ is the rest of $\gamma$. Since, we have assumed $\rho(xy)_{\{1,3\},\{1,3\}} \geq \rho(xy)_{\{1,2\},\{1,2\}}$ (and similarly for $\rho(yx)$), each loop at $\{1,3\}$ has greater weight than its counterpart at $\{1,2\}$. Since $\gamma$ is assumed to have maximal weight on the word $abuab[XY,YX]$, this means that the path $\lambda_{\{1,2\}}$ can be assumed to have length at most 1; it has length 0 if $\gamma_{\{1,2\},\{1,3\}}$ is labelled $XY$, and length 1 if $\gamma_{\{1,2\},\{1,3\}}$ is labelled $YX$.
\par Therefore, the edge $\gamma_{\{1,2\},\{1,3\}}$ is contained within the first two edges of $\gamma$ corresponding to the first two letters of $abuab$ and hence by the definition of matrix multiplication in $UT_{2^{[4]}}(\T)$ we have that
\[(abuab(XY,YX))_{\{1,2\},\{3,4\}} = (ab(XY,YX))_{\{1,2\},P} + (uab(XY,YX))_{P,\{3,4\}}\] 
for some $P \geq \{1,3\}$. Moreover, as each such path from $P$ to $\{3,4\}$ (and hence the path of maximal weight) has simple length at most 3, we can apply Lemma~\ref{tech2} to get that
\begin{align*}
(abuab(XY,YX))_{\{1,2\},\{3,4\}} &= (ab(XY,YX))_{\{1,2\},P} + (uab(XY,YX))_{P,\{3,4\}} \\
&\leq (ab(XY,YX))_{\{1,2\},P} + (vab(XY,YX))_{P,\{3,4\}} \\ 
&\leq (abvab(XY,YX))_{\{1,2\},\{3,4\}}
\end{align*}

\end{enumerate}
We can now apply a similar case analysis to a maximal weight path labelled by $abvab[XY,YX]$ to get that 
\[(abuab(XY,YX))_{\{1,2\},\{3,4\}} \geq (abvab(XY,YX))_{\{1,2\},\{3,4\}}.\]
Therefore, $(abuab(XY,YX))_{\{1,2\},\{3,4\}} = (abvab(XY,YX))_{\{1,2\},\{3,4\}}$, and by (\ref{first}) we can conclude that $abuab(xy,yx) = abvab(xy,yx)$.
\par A similar argument in the case where $\rho(xy)_{\{2,4\},\{2,4\}} \geq \rho(xy)_{\{3,4\}\{3,4\}}$, applies to show that $abuab(xy,yx) = abvab(xy,yx)$.
\end{proof}

\begin{cor}
There exists an identity satisfied by $\Pb_4$ not satisfied by $UT_5(\T)$.
\end{cor}
\begin{proof}
Let $u= ba^3b^3aba \ b \ ba^3b^3aba$ and $ v= ba^3b^3aba \ a \ ba^3b^3aba$. By Theorem~\ref{Zur}, we have that $u[ab,ba] = v[ab,ba]$ is an identity satisfied by $UT_4(\T)$. So by Theorem~\ref{plactic}, $abuab[ab,ba] = abvab[ab,ba]$ is satisfied by $\Pb_4$. However, $abab$ is a factor of $abuab$ but not of $abvab$. So, by Lemma~\ref{factor}, we have that there exists $A,B \in UT_5(\T)$ such that $abuab(AB,BA) \neq abvab(AB,BA)$, and thus $abuab[ab,ba] = abvab[ab,ba]$ is not satisfied by $UT_5(\T)$.
\end{proof}

\end{document}